\documentclass[11pt,twoside]{birkjour}
\usepackage{amsmath, amsthm, amscd, amsfonts, amssymb, graphicx}
\usepackage{enumerate}
\usepackage[colorlinks=true,
linkcolor=blue,
urlcolor=cyan,
citecolor=red]{hyperref}
\usepackage{mathrsfs}
\newtheorem{theorem}{Theorem}[section]

\newtheorem{remark}{Remark}[section]

\newtheorem{example}{Example}[section]

\numberwithin{equation}{section}

\begin{document}
\title{On the operator Jensen-Mercer  inequality}
\author{Hamid Reza Moradi}

\address{%
Young Researchers and Elite Club\\
Mashhad Branch\\
slamic Azad University\\
Mashhad, Iran}

\email{hrmoradi@mshdiau.ac.ir}

\author{Shigeru Furuichi}
\address{Department of Information Science\br
College of Humanities and Sciences\br
Nihon University\br
3-25-40, Sakurajyousui\br
Setagaya-ku, Tokyo\br
156-8550, Japan}
\email{furuichi@chs.nihon-u.ac.jp}
\thanks{The author (S.F.) was partially supported by JSPS KAKENHI Grant Number 16K05257.}
\author{Mohammad Sababheh}
\address{Department of Basic Sciences\br
Princess Sumaya University for Technology\br
Amman 11941\br
Jordan}
\email{sababheh@psut.edu.jo}
 \maketitle
\begin{abstract}
Mercer inequality for convex functions is a variant of Jensen's inequality, with an operator version that is still valid without operator convexity.

This paper is two folded. First, we present a Mercer-type inequality for operators without assuming convexity nor operator convexity. Yet, this form refines the known inequalities in the literature. Second, we present a log-convex version for operators. We then use these results to refine some inequalities related to quasi-arithmetic means of Mercer's type for operators.
\end{abstract}
\pagestyle{myheadings}
\markboth{\centerline {On the operator Jensen-Mercer  inequality}}
{\centerline {H.R. Moradi, S. Furuichi \& M. Sababheh}}
\bigskip
\bigskip
\section{\bf Introduction}
Recall that a function $f:I\subseteq \mathbb{R}\to \mathbb{R}$ is said to be convex on the interval $I$, if it satisfies the Jensen inequality
\begin{equation}\label{jensen}
f\left( \sum\limits_{i=1}^{n}{{{w}_{i}}{{x}_{i}}} \right)\le \sum\limits_{i=1}^{n}{{{w}_{i}}f\left( {{x}_{i}} \right)},
\end{equation}
 for all choices of positive scalars ${{w}_{1}},\ldots ,{{w}_{n}}$ with $\sum\nolimits_{i=1}^{n}{{{w}_{i}}}=1$ and ${{x}_{i}}\in I$. It is well known that this general form is equivalent to the same inequality when $n=2.$
 
 In 2003, Mercer found a variant of  \eqref{jensen}, which reads as follows.
\begin{theorem}{\bf (\cite[Theorem 1.2]{3})} 
If $f$ is a convex function on $\left[ m,M \right]$, then 
\begin{equation}\label{3}
f\left( M+m-\sum\limits_{i=1}^{n}{{{w}_{i}}{{x}_{i}}} \right)\le f\left( M \right)+f\left( m \right)-\sum\limits_{i=1}^{n}{{{w}_{i}}f\left( {{x}_{i}} \right)},
\end{equation} 
for all ${{x}_{i}}\in \left[ m,M \right]$ and all ${{w}_{i}}\in \left[ 0,1 \right]$ $\left( i=1,\ldots ,n \right)$ with $\sum\nolimits_{i=1}^{n}{{{w}_{i}}}=1$.    
\end{theorem}
There are many versions, variants and generalizations for the inequality \eqref{3}; see for example \cite{11, 4,1}.

It is customary in the field of Mathematical inequalities to extend scalar inequalities, like \eqref{jensen} and \eqref{3}, to operators on Hilbert spaces.
For this end, we adopt the following notations. Let $\mathscr{H}$ and $\mathscr{K}$ be Hilbert spaces, $\mathbb{B}\left( \mathscr{H} \right)$ and $\mathbb{B}\left( \mathscr{K} \right)$ be the ${{C}^{*}}$-algebras of all bounded operators on the appropriate Hilbert space. An operator $A\in \mathscr{H}$ is called self-adjoint if $A=A^*$, where $A^*$ denotes the adjoint operator of $A$. If $A\in \mathscr{H},$ the notation $A\geq 0$ will be used to declare that $A$ is positive, in the sense that $\left<Ax,x\right>\geq 0$ for all $x\in \mathscr{H}.$ If $\left<Ax,x\right>>0$ for all non zero $x\in \mathscr{H},$ we write $A>0$, and we say than that $A$ is positive definite. On the class of self-adjoint operators, the $\leq $ partial order relation is well known, where we write $A\leq B$ if $B-A\geq 0$, when $A,B$ are self-adjoint.

In studying operator inequalities, the notion of spectrum cannot be avoided. If $A\in \mathscr{H},$ the spectrum of $A$ is defined by
$$\sigma(A)=\{\lambda\in \mathbb{C}: A-\lambda {\mathbf{1}}_{\mathscr{H}}\;{\text{is\;not\;invertible}}\},$$ where ${\mathbf{1}}_{\mathscr{H}}$ denotes the identity operator on $\mathscr{H}.$ Finally, in these terminologies, a linear map $\Phi :\mathbb{B}\left( \mathscr{H} \right)\to \mathbb{B}\left( \mathscr{K} \right)$ is said to be positive if $\Phi \left( A \right)\ge 0$ whenever $A\ge 0$ and $\Phi $ is called unital if $\Phi \left( {{\mathbf{1}}_{\mathscr{H}}} \right)={{\mathbf{1}}_{\mathscr{K}}}$.

 Recall that a continuous function $f:I\to\mathbb{R}$ is said to be operator convex if 
$$f\left(\frac{A+B}{2}\right)\leq \frac{f(A)+f(B)}{2}$$
for all self-adjoint $A,B\in \mathbb{B}\left( \mathscr{H} \right)$   and $\sigma(A),\sigma(B)\subset I.$ This is equivalent to the Jensen operator inequality, valid for the self-adjoint operators $A_i$ whose spectra are in the interval $I$, 
\begin{equation}\label{jensen_op_intro}
f\left(\sum_{i=1}^{n}w_iA_i\right)\leq \sum_{i=1}^{n}w_if(A_i),\,\,\,\,w_i>0,\,\,\, \sum_{i=1}^{n}w_i=1.
\end{equation}
It is evident that a convex function is not necessarily operator convex, and the function $f(x)=x^4$ provides such an example. Thus, a convex function does not necessarily satisfy the operator Jensen inequality \eqref{jensen_op_intro}. However, it turns out that a convex function satisfies the following operator version of the Mercer inequality \eqref{3}.
\begin{theorem}\label{thbb}{\bf (\cite[Theorem 1]{2})} Let ${{A}_{1}},\ldots ,{{A}_{n}}\in \mathbb{B}\left( \mathscr{H} \right)$ be self-adjoint operators with spectra in $\left[ m,M \right]$ and let ${{\Phi }_{1}},\ldots ,{{\Phi }_{n}}:\mathbb{B}\left( \mathscr{H} \right)\to \mathbb{B}\left( \mathscr{K} \right)$ be positive linear maps with $\sum\nolimits_{i=1}^{n}{{{\Phi }_{i}}\left( {{\mathbf{1}}_{\mathscr{H}}} \right)}={{\mathbf{1}}_{\mathscr{K}}}$. If $f:\left[ m,M \right]\subseteq \mathbb{R}\to \mathbb{R}$ is a convex function, then 
\begin{equation}\label{15}
f\left( \left( M+m \right){{1}_{K}}-\sum\limits_{i=1}^{n}{{{\Phi }_{i}}\left( {{A}_{i}} \right)} \right)\le \left( f\left( M \right)+f\left( m \right) \right){{1}_{K}}-\sum\limits_{i=1}^{n}{{{\Phi }_{i}}\left( f\left( {{A}_{i}} \right) \right)}.
\end{equation}
\end{theorem}
Further, in the same reference, the following series of inequalities was proved
\begin{align}
& f\left( \left( M+m \right){{\mathbf{1}}_{\mathscr{K}}}-\sum\limits_{i=1}^{n}{{{\Phi }_{i}}\left( {{A}_{i}} \right)} \right) \le \left( f\left( M \right)+f\left( m \right) \right){{\mathbf{1}}_{\mathscr{K}}}\nonumber\\ 
&+\frac{\sum\nolimits_{i=1}^{n}{{{\Phi }_{i}}\left( {{A}_{i}} \right)}-M{{\mathbf{1}}_{\mathscr{K}}}}{M-m}f\left( m \right)+\frac{m{{\mathbf{1}}_{\mathscr{K}}}-\sum\nolimits_{i=1}^{n}{{{\Phi }_{i}}\left( {{A}_{i}} \right)}}{M-m}f\left( M \right) \label{m}\\ 
& \le \left( f\left( M \right)+f\left( m \right) \right){{\mathbf{1}}_{\mathscr{K}}}-\sum\limits_{i=1}^{n}{{{\Phi }_{i}}\left( f\left( {{A}_{i}} \right) \right)}\nonumber. 
\end{align}
Later, related and analogous results have been established in \cite{10, 15, 14}.

Our main goal of this article is to present a refinement of the operator inequality \eqref{15} without using convexity of $f$. Rather, 
using the idea by Mi\'ci\'c et al. \cite{8}, 
we assume a boundedness condition on $f''$. Then a discussion of log-convex version of Mercer's operator  inequality will be presented.
  
\section{\bf Main Results}
In this section we present our main results in two parts. In the first part, we discuss the twice differentiable case, then we discuss the log-convex case.

\subsection{\bf Twice Differentiable Functions}\label{s2}
We begin with the non-convex version of Theorem \ref{thbb}. We use the following symbol in this paper.
\begin{enumerate}[(i)]
\item $\mathbf{A}^{sa}=\left( {{A}_{1}},\ldots ,{{A}_{n}} \right)$, where ${{A}_{i}}\in \mathbb{B}\left( \mathscr{H} \right)$ are self-adjoint operators with $\sigma \left( {{A}_{i}} \right)\subseteq \left[ m,M \right]$ for some scalars $0<m<M$.
\item $\mathbf{\Phi}^+ =\left( {{\Phi }_{1}},\ldots ,{{\Phi }_{n}} \right)$, where ${{\Phi }_{i}}:\mathbb{B}\left( \mathscr{H} \right)\to \mathbb{B}\left( \mathscr{K} \right)$ are positive linear maps.
\end{enumerate}
\begin{theorem}\label{tha}
Let ${{A}_{1}},\ldots ,{{A}_{n}}\in \mathbb{B}\left( \mathscr{H} \right)$ be self-adjoint operators with spectra in $\left[ m,M \right]$ and let ${{\Phi }_{1}},\ldots ,{{\Phi }_{n}}:\mathbb{B}\left( \mathscr{H} \right)\to \mathbb{B}\left( \mathscr{K} \right)$ be positive linear maps with $\sum\nolimits_{i=1}^{n}{{{\Phi }_{i}}\left( {{\mathbf{1}}_{\mathscr{H}}} \right)}={{\mathbf{1}}_{\mathscr{K}}}$. If $f:\left[ m,M \right]\subseteq \mathbb{R}\to \mathbb{R}$ is a continuous twice differentiable function such that $\alpha \le f''\le \beta $ with $\alpha ,\beta \in \mathbb{R}$, then
\begin{eqnarray}
&& \left( f\left( M \right)+f\left( m \right) \right){{\mathbf{1}}_{\mathscr{K}}}-\sum\limits_{i=1}^{n}{{{\Phi }_{i}}\left( f\left( {{A}_{i}} \right) \right)} -\beta J(m,M,\mathbf{A}^{sa},\mathbf{\Phi}^+ ) \nonumber\\ 
&& \le f\left( \left( M+m \right){{\mathbf{1}}_{\mathscr{K}}}-\sum\limits_{i=1}^{n}{{{\Phi }_{i}}\left( {{A}_{i}} \right)} \right) \label{0001}\\ 
&& \le \left( f\left( M \right)+f\left( m \right) \right){{\mathbf{1}}_{\mathscr{K}}}-\sum\limits_{i=1}^{n}{{{\Phi }_{i}}\left( f\left( {{A}_{i}} \right) \right)} -\alpha J(m,M,\mathbf{A}^{sa},\mathbf{\Phi}^+ ), \label{0002}
\end{eqnarray}
where 
\begin{eqnarray}
J(m,M,\mathbf{A}^{sa},\mathbf{\Phi}^+)&&\hspace*{-5mm}:=\left( M+m \right)\sum\limits_{i=1}^{n}{{{\Phi }_{i}}\left( {{A}_{i}} \right)}-Mm{{\mathbf{1}}_{\mathscr{K}}}\nonumber \\
&&-\frac{1}{2}\left( {{\left( \sum\limits_{i=1}^{n}{{{\Phi }_{i}}\left( {{A}_{i}} \right)} \right)}^{2}}+\sum\limits_{i=1}^{n}{{{\Phi }_{i}}\left( A_{i}^{2} \right)} \right)\geq 0. 
\end{eqnarray}
\end{theorem}
\begin{proof}
Notice that for any convex function $f$ and $m\le t\le M$, we have
\begin{equation}\label{2}
f\left( t \right)=f\left( \frac{M-t}{M-m}m+\frac{t-m}{M-m}M \right)\le L_{f}\left( t \right),
\end{equation}
where 
\begin{equation}\label{l}
L_{f}\left( t \right):= \frac{M-t}{M-m}f\left( m \right)+\frac{t-m}{M-m}f\left( M \right).
\end{equation}
Letting ${{g}_{\alpha }}\left( t \right):= f\left( t \right)-\frac{\alpha }{2}{{t}^{2}}$ $\left( m\le t\le M \right)$, we observe that $g$ is convex noting the assumption $\alpha\leq f''$. Applying \eqref{2} to the function $g$, we have $g(t)\leq L_g(t)$, which leads to   
\begin{equation}\label{1}
f\left( t \right)\le L_{f}\left( t \right)-\frac{\alpha }{2}\left\{ \left( M+m \right)t-Mm-{{t}^{2}} \right\}.
\end{equation}
Since $m\le M+m-t\le M$, we can replace $t$ in \eqref{1} with $M+m-t$, to get
	\[f\left( M+m-t \right)\le {{L}_{0}}\left( t \right)-\frac{\alpha }{2}\left\{ \left( M+m \right)t-Mm-{{t}^{2}} \right\},\] 
where
\[{{L}_{0}}\left( t \right):= L\left( M+m-t \right)=f\left( M \right)+f\left( m \right)-L_{f}\left( t \right).\]
Using functional calculus for the operator $m{{\mathbf{1}}_{\mathscr{K}}}\le \sum\nolimits_{i=1}^{n}{{{\Phi }_{i}}\left( {{A}_{i}} \right)}\le M{{\mathbf{1}}_{\mathscr{K}}}$, we infer that
\begin{equation}\label{m1}
\begin{aligned}
  & f\left( \left( M+m \right){{\mathbf{1}}_{\mathscr{K}}}-\sum\limits_{i=1}^{n}{{{\Phi }_{i}}\left( {{A}_{i}} \right)} \right)\le {{L}_{0}}\left( \sum\limits_{i=1}^{n}{{{\Phi }_{i}}\left( {{A}_{i}} \right)} \right) \\ 
 &\quad -\frac{\alpha }{2}\left\{ \left( M+m \right)\sum\limits_{i=1}^{n}{{{\Phi }_{i}}\left( {{A}_{i}} \right)}-Mm{{\mathbf{1}}_{\mathscr{K}}}-{{\left( \sum\limits_{i=1}^{n}{{{\Phi }_{i}}\left( {{A}_{i}} \right)} \right)}^{2}} \right\}.  
\end{aligned}
\end{equation}
On the other hand, by applying functional calculus for the operator $m{{\mathbf{1}}_{\mathscr{H}}}\le {{A}_{i}}\le M{{\mathbf{1}}_{\mathscr{H}}}$ in \eqref{1}, we get
\[f\left( {{A}_{i}} \right)\le L_{f}\left( {{A}_{i}} \right)-\frac{\alpha }{2}\left\{ \left( M+m \right){{A}_{i}}-Mm{{\mathbf{1}}_{\mathscr{H}}}-A_{i}^{2} \right\}.\]
Applying the positive linear maps ${{\Phi }_{i}}$ and adding in the last inequality yield 
\begin{eqnarray}
&&\sum\limits_{i=1}^{n}{{{\Phi }_{i}}\left( f\left( {{A}_{i}} \right) \right)}\le L_0\left( \sum\limits_{i=1}^{n}{{{\Phi }_{i}}\left( {{A}_{i}} \right)} \right) \nonumber \\
&& -\frac{\alpha }{2}\left\{ \left( M+m \right)\sum\limits_{i=1}^{n}{{{\Phi }_{i}}\left( {{A}_{i}} \right)}-Mm{{\mathbf{1}}_{\mathscr{K}}}-\sum\limits_{i=1}^{n}{{{\Phi }_{i}}\left( A_{i}^{2} \right)} \right\}.\label{mm1}
\end{eqnarray}
Combining the two inequalities \eqref{m1} and \eqref{mm1}, we get \eqref{0002}.

Finally we give the proof of $J(m,M,\Phi_i,A_i) \geq 0$.
Since $m {{\mathbf{1}}_{\mathscr{H}}}\leq A_i\leq M {{\mathbf{1}}_{\mathscr{H}}}$, we have $(M  {{\mathbf{1}}_{\mathscr{H}}}-A_i)(A_i-m  {{\mathbf{1}}_{\mathscr{H}}}) \geq 0$ which implies
$
(M+m)A_i-m M {{\mathbf{1}}_{\mathscr{H}}} -A_i^2\geq 0$. Thus we have
$(M+m)\Phi(A_i)-m M \Phi({{\mathbf{1}}_{\mathscr{H}}}) -\Phi(A_i^2)\geq 0.$
Taking a summation on $i=1,\cdots ,n$ of this inequality with taking an account for $\sum_{i=1}^{n}\Phi_
i({\mathbf{1}}_{\mathscr{H}})={\mathbf{1}}_{\mathscr{K}}$, we obtain
\begin{equation}\label{needed_1_remark}
\left( M+m \right)\sum\limits_{i=1}^{n}{{{\Phi }_{i}}\left( {{A}_{i}} \right)}-Mm{{\mathbf{1}}_{\mathscr{K}}}-\sum_{i=1}^{n}\Phi(A_i^2)\geq 0.
\end{equation}
Further, noting that $m\leq \sum_{i=1}^{n}\Phi_i(A_i)\leq M,$ we also have
\begin{equation}\label{needed_2_remark}
(M+m)\sum_{i=1}^{n}\Phi_i(A_i)-m M{\mathbf{1}}_{\mathscr{K}}-\left(\sum_{i=1}^{n}\Phi_i(A_i)\right)^2\geq 0.
\end{equation}
Adding \eqref{needed_1_remark} and \eqref{needed_2_remark} and dividing by $2$, we obtain $J(m,M,\mathbf{A}^{sa},\mathbf{\Phi}^+)\ge 0.$

The inequality \eqref{0001} follows similarly by taking into account that
\[L_{f}\left( t \right)-\frac{\beta }{2}\left\{ \left( M+m \right)t-Mm-{{t}^{2}} \right\}\le f\left( t \right),\qquad\text{ }m\le t\le M.\]
The details are left to the reader. This completes the proof.
\end{proof}

\medskip

In the following example, we present the advantage of using twice differentiable functions in Theorem \ref{tha}. 
\begin{example}
Let $f\left( t \right)=\sin t\text{ }\left( 0\le t\le 2\pi  \right)$, $A=\left( \begin{matrix}
\frac{\pi }{4} & 0  \\
0 & \frac{\pi }{2}  \\
\end{matrix} \right)$ and $\Phi \left( A \right)=\frac{1}{2}Tr\left[ A \right]$. Actually the function $f(t)=\sin t$ is concave on $[0,\pi]$. Letting $m=\frac{\pi }{4}$ and $M=\frac{\pi }{2}$, we obtain 
\[0.9238\approx f\left( \left( M+m \right)-\Phi \left( A \right) \right)\nless f\left( M \right)+f\left( m \right)-\Phi \left( f\left( A \right) \right)\approx 0.8535.\]
That is, \eqref{15} may fail without the convexity assumption. However, by considering the weaker assumptions assumed in Theorem \ref{tha}, we get
\[\begin{aligned}
 0.9238&\approx f\left( \left( M+m \right)-\Phi \left( A \right) \right) \\ 
& \lneqq f\left( M \right)+f\left( m \right)-\Phi \left( f\left( A \right) \right)\\
& -\alpha \left\{ \left( M+m \right)\Phi \left( A \right)-Mm-\frac{1}{2}\left\{ \Phi {{\left( A \right)}^{2}}+\Phi \left( {{A}^{2}} \right) \right\} \right\}\approx 0.9306, \\ 
\end{aligned}\]
since $f''(t)=-\sin t$ which gives $\alpha=-1.$
\end{example}

To better understand the relation between Theorems \ref{thbb} and \ref{tha}, we present the following remark, where we clarify how the first theorem is retrieved from the second.

\begin{remark}
The inequality \eqref{0002} in Theorem \ref{tha} with an assumption on a twice differentiable function $f$ such that $\alpha \leq f''\leq \beta$ for $\alpha,\beta \in \mathbb{R}$ gives a better upper bound of $$f\left( \left( M+m \right){{\mathbf{1}}_{\mathscr{K}}}-\sum\limits_{i=1}^{n}{{{\Phi }_{i}}\left( {{A}_{i}} \right)} \right)$$ than that in \eqref{15}, since $J(m,M,\mathbf{A}^{sa},\mathbf{\Phi}^+)\geq 0$, if we take $\alpha \geq 0$. Additionally to this result, we obtained a reverse type inequality \eqref{0001} which gives a lower bound of $$f\left( \left( M+m \right){{\mathbf{1}}_{\mathscr{K}}}-\sum\limits_{i=1}^{n}{{{\Phi }_{i}}\left( {{A}_{i}} \right)} \right).$$
\end{remark}

\subsection{\bf Log-convex Functions}\label{s3}
We conclude this section by presenting Mercer-type operator inequalities for log-convex functions. Recall that a positive function defined on an interval $I$ (or, more generally, on a convex subset of some vector space) is called  {\it $log $-convex} if $\log f\left( x \right)$ is a convex function of $x$. We observe that such functions satisfy the elementary inequality

\[f\left( \left( 1-v \right)a+vb \right)\le {{\left[ f\left( a \right) \right]}^{1-v}}{{\left[ f\left( b \right) \right]}^{v}},\qquad \text{ }0\le v\le 1\] 
for any $a,b\in I$. $f$ is called {\it $log $-concave} if the inequality above is reversed (that is, when $\frac{1}{f}$ is $\log $-convex).
By virtue of the arithmetic-geometric mean inequality, we have
\begin{equation}\label{08}
f\left( \left( 1-v \right)a+vb \right)\le {{\left[ f\left( a \right) \right]}^{1-v}}{{\left[ f\left( b \right) \right]}^{v}}\le \left( 1-v \right)f\left( a \right)+vf\left( b \right),
\end{equation}
which implies convexity of log-convex functions. This double inequality is of special interest since \eqref{08} can be written as
\begin{equation}\label{008}
f\left( t \right)\le {{\left[ f\left( m \right) \right]}^{\frac{M-t}{M-m}}}{{\left[ f\left( M \right) \right]}^{\frac{t-m}{M-m}}}\le L_f\left( t \right),\qquad\text{ }m\le t\le M
\end{equation}
where $L_f\left( t \right)$ is as in \eqref{l}.

Manipulating the inequality \eqref{008}, we have the following extension of Theorem \ref{thbb} to the context of log-convex functions. The proof is left to the reader.
\begin{theorem}\label{010}
Let all the assumptions of Theorem \ref{thbb} hold except that $f:\left[ m,M \right]\to \left( 0,\infty  \right)$ is log-convex. Then
\begin{equation}\label{011}
\begin{aligned}
 f\left( \left( M+m \right){{\mathbf{1}}_{\mathscr{K}}}-\sum\limits_{i=1}^{n}{{{\Phi }_{i}}\left( {{A}_{i}} \right)} \right)&\le {{\left[ f\left( m \right) \right]}^{\frac{\sum\nolimits_{i=1}^{n}{{{\Phi }_{i}}\left( {{A}_{i}} \right)-m{{\mathbf{1}}_{\mathscr{K}}}}}{M-m}}}{{\left[ f\left( M \right) \right]}^{\frac{M{{\mathbf{1}}_{\mathscr{K}}}-\sum\nolimits_{i=1}^{n}{{{\Phi }_{i}}\left( {{A}_{i}} \right)}}{M-m}}} \\ 
& \le \left( f\left( M \right)+f\left( m \right) \right){{\mathbf{1}}_{\mathscr{K}}}-\sum\limits_{i=1}^{n}{{{\Phi }_{i}}\left( f\left( {{A}_{i}} \right) \right)}.  
\end{aligned}
\end{equation}
\end{theorem}
\section{\bf Applications}\label{s4}
In this section, we present some applications of the main results that we have shown so far. First, we review and introduce the notations.
\begin{enumerate}[(i)]
\item $\mathbf{A}^+=\left( {{A}_{1}},\ldots ,{{A}_{n}} \right)$, where ${{A}_{i}}\in \mathbb{B}\left( \mathscr{H} \right)$ are positive invertible operators with $\sigma \left( {{A}_{i}} \right)\subseteq \left[ m,M \right]$ for some scalars $0<m<M$.
\item $\mathbf{\Phi}^+ =\left( {{\Phi }_{1}},\ldots ,{{\Phi }_{n}} \right)$, where ${{\Phi }_{i}}:\mathbb{B}\left( \mathscr{H} \right)\to \mathbb{B}\left( \mathscr{K} \right)$ are positive linear maps.
\item  $C\left( \left[ m,M \right] \right)$ is the set of all real valued continuous functions on an interval $\left[ m,M \right]$.
\end{enumerate}
 We also need to remind the reader that a function $f\in C\left( \left[ m,M \right] \right)$ is called operator monotone increasing (or operator increasing for short) if $f\left( A \right)\le f\left( B \right)$ whenever $A,B$ are self-adjoint operators with spectra in $[m,M]$ and such that $A\leq B.$ That is, when $f$ preserves the order of self-adjoint operator. A function $f\in C\left( \left[ m,M \right] \right)$ is said to be operator decreasing if $-f$ is operator monotone.

The so called operator quasi-arithmetic mean of Mercer's type was defined in \cite{2} as follows:
\[{{\widetilde{M}}_{\varphi }}\left( \mathbf{A}^+,\mathbf{\Phi}^+  \right):= {{\varphi }^{-1}}\left( \left( \varphi \left( M \right)+\varphi \left( m \right) \right){{\mathbf{1}}_{\mathscr{K}}}-\sum\limits_{i=1}^{n}{{{\Phi }_{i}}\left( \varphi \left( {{A}_{i}} \right) \right)} \right).\]
In this reference, the following result was shown.
\begin{theorem}
Let $\varphi ,\psi \in C\left( \left[ m,M \right] \right)$ be two strictly monotonic functions.
\begin{itemize}
	\item[(i)] If either $\psi \circ {{\varphi }^{-1}}$ is convex and ${{\psi }^{-1}}$ is operator increasing, or $\psi \circ {{\varphi }^{-1}}$ is concave and ${{\psi }^{-1}}$ is operator decreasing, then 
	\begin{equation}\label{3.1}
{{\widetilde{M}}_{\varphi }}\left( \mathbf{A}^+,\mathbf{\Phi}^+  \right)\le {{\widetilde{M}}_{\psi }}\left( \mathbf{A}^+,\mathbf{\Phi}^+  \right).
	\end{equation}
	\item[(ii)] If either $\psi \circ {{\varphi }^{-1}}$ is concave and ${{\psi }^{-1}}$ is operator increasing, or $\psi \circ {{\varphi }^{-1}}$ is convex and ${{\psi }^{-1}}$ is operator decreasing, then the inequality in \eqref{3.1} is reversed.
\end{itemize}
\end{theorem}

By virtue of Theorem \ref{tha}, we have the following extension of this result.
\begin{theorem}\label{th3}
Let $\varphi ,\psi \in C\left( \left[ m,M \right] \right)$ be two strictly monotonic functions and $\psi \circ {{\varphi }^{-1}}$ is twice differentiable function. 
\begin{itemize}
	\item[(i)] If $\alpha \le {{\left( \psi \circ {{\varphi }^{-1}} \right)}^{''}}$ with $\alpha \in \mathbb{R}$ and ${{\psi }^{-1}}$ is operator monotone, then 
	\begin{equation}\label{o}
	{{\widetilde{M}}_{\varphi }}\left( \mathbf{A}^+,\mathbf{\Phi}^+  \right)\le {{\psi }^{-1}}\left\{ \psi \left( {{\widetilde{M}}_{\psi }}\left( \mathbf{A}^+,\mathbf{\Phi}^+  \right) \right)-\alpha K \left( m,M,\varphi ,\mathbf{A}^+,\mathbf{\Phi}^+  \right) \right\},
	\end{equation}
	where
	\[\begin{aligned}
	K \left( m,M,\varphi ,\mathbf{A}^+,\mathbf{\Phi}^+  \right)&:= \left( \varphi \left( M \right)+\varphi \left( m \right) \right)\sum\limits_{i=1}^{n}{{{\Phi }_{i}}\left( \varphi \left( {{A}_{i}} \right) \right)}-\varphi \left( M \right)\varphi \left( m \right){{\mathbf{1}}_{\mathscr{K}}} \\ 
	&\quad -\frac{1}{2}\left( {{\left( \sum\limits_{i=1}^{n}{{{\Phi }_{i}}\left( \varphi \left( {{A}_{i}} \right) \right)} \right)}^{2}}+\sum\limits_{i=1}^{n}{{{\Phi }_{i}}\left( \varphi {{\left( {{A}_{i}} \right)}^{2}} \right)} \right).  
	\end{aligned}\]
	\item[(ii)] If ${{\left( \psi \circ {{\varphi }^{-1}} \right)}^{''}}\le \beta $ with $\beta \in \mathbb{R}$ and ${{\psi }^{-1}}$ is operator monotone, then the reverse inequality is valid in \eqref{o} with $\beta $ instead of $\alpha $.
\end{itemize}
\end{theorem}
\begin{proof}
Let $f=\psi \circ {{\varphi }^{-1}}$ in \eqref{0002} and replace ${{A}_{i}}$, $m$ and $M$ with $\varphi \left( {{A}_{i}} \right)$, $\varphi \left( m \right)$ and $\varphi \left( M \right)$ respectively. This implies
\[\psi \left( {{\widetilde{M}}_{\varphi }}\left( \mathbf{A}^+,\mathbf{\Phi}^+  \right) \right)\le \psi \left( {{\widetilde{M}}_{\psi }}\left( \mathbf{A}^+,\mathbf{\Phi}^+  \right) \right)-\alpha K \left( m,M,\varphi ,\mathbf{A}^+,\mathbf{\Phi}^+  \right).\]
Since ${{\psi }^{-1}}$ is operator monotone, the first conclusion follows immediately. The other case follows in a similar manner from \eqref{0001}. 
\end{proof}

\medskip

Similarly, Theorem \ref{010} implies the following version.
\begin{theorem}\label{th4}
Let $\varphi ,\psi \in C\left( \left[ m,M \right] \right)$ be two strictly monotonic functions. If $\psi \circ {{\varphi }^{-1}}$ is log-convex function and ${{\psi }^{-1}}$ is operator increasing, then
\[\begin{aligned}
{{\widetilde{M}}_{\varphi }}\left( \mathbf{A}^+,\mathbf{\Phi}^+  \right)&\le {{\psi }^{-1}}\left\{ {{\left[ \psi \left( m \right) \right]}^{\frac{\sum\nolimits_{i=1}^{n}{{{\Phi }_{i}}\left( \varphi \left( {{A}_{i}} \right) \right)}-\varphi \left( m \right){{\mathbf{1}}_{\mathscr{K}}}}{\varphi \left( M \right)-\varphi \left( m \right)}}}{{\left[ \psi \left( M \right) \right]}^{\frac{\varphi \left( M \right){{\mathbf{1}}_{\mathscr{K}}}-\sum\nolimits_{i=1}^{n}{{{\Phi }_{i}}\left( \varphi \left( {{A}_{i}} \right) \right)}}{\varphi \left( M \right)-\varphi \left( m \right)}}} \right\}\\
&\le {{\widetilde{M}}_{\psi }}\left( \mathbf{A}^+,\mathbf{\Phi}^+  \right).
\end{aligned}\]
\end{theorem}
\begin{remark}
By choosing appropriate functions $\varphi $ and $\psi $, and making suitable substitutions, the above results imply some improvements of certain inequalities governing operator power mean of Mercer's type. We leave the details of this idea to the interested reader as an application of our main results.
\end{remark}

\medskip

In the end of the article, we show the example such that there is no relationship between inequalities in Theorems \ref{th3} and \ref{th4}. Here, we restrict ourselves to the power function $f\left( t \right)={{t}^{p}}$ with $p<0$.
\begin{example}
	It is sufficient to compare \eqref{1} and the first inequality of \eqref{008}. We take $m=1$ and $M=3$. Setting
\[\begin{aligned}
 g(t) &=  \frac{M-t}{M-m} m^p +\frac{t-m}{M-m} M^p -\frac{p(p-1)M^{p-2}}{2} \left\{ (M+m) t - M m -t^2\right\} \\
 &-\left( m^{\frac{M-t}{M-m}} M^{\frac{t-m}{M-m}}\right)^p.
 \end{aligned}\]
	Some calculations show that $g(2) \approx -0.0052909$ when $p=-0.2$, while $g(2) \approx 0.0522794$ when $p=-1$.
	We thus conclude that there is no ordering between the RHS of inequality in \eqref{1} and the RHS of first inequality of \eqref{008}.
\end{example}

\medskip

\noindent{\bf Acknowledgements.}
The authors would like to thank the referees for their careful and insightful comments to improve our manuscript.
 The authors are also grateful to Dr. Trung Hoa Dinh for fruitful discussion and revising the manuscript.
\bibliographystyle{alpha}

%
%
%
%
%
%
%

\end{document}